\documentclass[11pt]{article}                  
\usepackage{graphicx}

\usepackage{amsmath,amsfonts,amssymb,amsthm}
\usepackage{natbib}
\usepackage[sans]{dsfont}
\usepackage{mathrsfs}
\usepackage{nicefrac}
\newtheorem{theorem}{Theorem}[section]

\newtheorem{definition}[theorem]{Definition}
\newtheorem{lemma}[theorem]{Lemma}

\newtheorem{remark}[theorem]{Remark}

\newcommand{\refp}[1]{(\ref{#1})}

\newcommand{\bs}[1]{\boldsymbol{#1}}
\newcommand{\ud}[1]{\, \mathrm{d}#1}
\newcommand{\IND}[1] {{ \mathds{1}_{ #1 }} }

\newcommand{\R}{\mathbb{R}}

\newcommand{\Dom}{\text{Dom}}

\renewcommand{\L}{\mathcal{L}}
\newcommand{\D}{\mathcal{D}}
\newcommand{\A}{\mathcal{A}}
\newcommand{\E}{\mathcal{E}}

\newcommand{\C}{\mathcal{C}}

\renewcommand{\l}{\mathsf{\ell}}

\newcommand{\bphi}{\boldsymbol{\phi}}
\newcommand{\n}{\boldsymbol{n}}

\newcommand{\vp}{\varphi}

\renewcommand{\geq}{\geqslant}

\renewcommand{\d}{\partial}

\newcommand{\deriv}[3][]{\frac{\ud^{#1} \hspace{-0.3mm} #2}{\ud{#3}^{#1}}}
\newcommand{\pderiv}[3][]{\frac{\d^{#1} \hspace{-0.1mm} #2}{\d{#3}^{#1}}}

\begin{document}

\title{Green's functions for Sturm-Liouville problems on directed tree graphs.}


\author{Jorge M Ramirez \footnote{Universidad Nacional de Colombia, Sede Medellin. email: jmramirezo@unal.edu.co}}

\maketitle

\begin{abstract}
Let $\Gamma$ be geometric tree graph with $m$ edges and consider the second order Sturm-Liouville operator $\L[u]=(-pu')'+qu$ acting on functions that are continuous on all of $\Gamma$, and twice continuously differentiable in the interior of each edge. The functions $p$ and $q$ are assumed uniformly continuous on each edge, and $p$ strictly positive on $\Gamma$. The problem is to find a solution $f:\Gamma \to \R$ to the problem $\L[f] = h$ with $2m$ additional conditions at the nodes of $\Gamma$. These node conditions include continuity at internal nodes, and jump conditions on the derivatives of $f$ with respect to a positive measure $\rho$. Node conditions are given in the form of linear functionals $\l_1,\dots,\l_{2m}$ acting on the space of admissible functions. A novel formula is given for the Green's function $G:\Gamma\times \Gamma \to \R$ associated to this problem. Namely, the solution to the semi-homogenous problem $\L[f] = h$, $\l_i[f] =0$ for $i=1,\dots,2m$ is given by $f(x) = \int_\Gamma G(x,y) h(y) \ud \rho$.   
\end{abstract}

\section{Introduction}

The Sturm-Liouville differential operator
\begin{equation}\label{SLP}
\L[f] = -(pf')' + q f  
\end{equation}
on an interval, appears in the analysis of many different types of models in the natural sciences. The problem $\L[f] =h$ or $\L[f] = \nu q f$ together with appropriate boundary conditions arise when considering Kirchoff's law in electrical circuits, the balance of tension in a elastic string, or the steady state temperature in a heated rod (see for example \cite{Kreyszig:1999fk,Hjortso:2009uq,Guenther:1996fk}). A more complete review of the mathematical theory can be found in \cite{Zettl:2005kx}.

The extension of operator \refp{SLP} to the case of a domain composed of intervals arranged in a graph has received recent attention for the last thirty years (see for example \cite{Merkov:1985vn, Roth:1984ys, VonBelow:1988p3428}. A complete bibliographical review with historical notes can be found in \cite{Pokornyi:2004p3313}.

\subsection{Physical motivation}

The application example that follows serves as the motivation for the present study, and arises from a problem in stability of populations of organisms in river networks \citep{Ramirez:2011zr}.

When considering the dispersion of solutes or organisms in a river or stream, one might consider the following advection-diffusion model \begin{equation}
\pderiv{u}{t} = D \pderiv[2]{u}{x} - v \pderiv{u}{x}.
\end{equation}
Here $u(x,t)$ is the concentration per unit length of the dispersed quantity at position $x \in [0,l]$, and time $t>0$. The coefficients $D$ and $v$ are constant and strictly positive and denote the diffusivity and water velocity respectively (see \cite{Lutscher:2006p1843} for a justification of this model). 

Consider now a mathematical model for the dispersion of the same quantity in a collection of streams arranged in river network. One then might consider a domain in the form of a tree graph $\Gamma$ where each stream in the network corresponds to an edge $e$ of $\Gamma$, and the stream junctions and boundary points are the nodes of the graph. An edge $e$ can be parametrized as the interval $(0,l_e)$, the point $x=0$ corresponding to the downstream node of edge $e$. We denote by $\Gamma$ the union of all edges, and by $\bar\Gamma$ the set composed of $\Gamma$ and the nodes of the graph. The most downstream point of the network (i.e. its outlet) is the root node of $\bar\Gamma$ and is denoted by $\bphi$. Let $u:\Gamma \to \R$ be the longitudinal concentration, with $u_e$ denoting the restriction of $u$ to edge $e$. Then $u$ satisfies the following evolution equation
\begin{equation}
\pderiv{u}{t} = \A[u], \quad u(x,0) = f(x),
\end{equation}
where the differential operator  $\A$ is given on each edge by
\begin{equation}\label{A}
\A[h]\Big|_e := D_e h'' - v_e h'.
\end{equation}
The operator $\A$ acts on functions that satisfy certain regularity conditions inside each edge, but more interestingly, one must also specify conditions at the nodes of the tree graph. For the application in river networks, an internal node $\n$ is located where edges $e_1$, $e_2$ merge to form edge $e_0$. Since all nodes are assumed to be oriented downstream, one can talk about the value of $u_e$ and $u_e'$ at node $\n$ by taking the appropriate one-sided limits. In the particular application of dispersion on $\Gamma$, one requires continuity of the concentration:
\begin{equation}\label{contNode}
 u_{e_0}(\n) = u_{e_1}(\n) = u_{e_2}(\n), \quad i,j=1,\dots,k,
\end{equation}
and a flux matching condition
\begin{equation}\label{fluxNode}
 \rho_1 u_{e_1}(\n) + \rho_2 u_{e_2}(\n) = \rho_0 u_{e_0}(\n)
\end{equation}
for some nonzero coefficients $\rho_e$, $e\subset \Gamma$. The set of nodes that have a single incident edge is called the boundary of $\Gamma$, and $\d \Gamma$ is used for its notation. The phenomena of dispersion typically imposes  Dirichlet (absorbing) boundary conditions at the root node, and Neumann (reflecting) condition at all other (upstream) nodes. Namely,
\begin{equation}\label{boundCond}
u(\bphi) = 0, \quad u'(\bs{n}) = 0, \; \bs{n} \in \d \Gamma \setminus \{\bphi\}.
\end{equation}

Consider the following integrating factors
\begin{equation}\label{pq}
p(x) := \exp \left\{ -\int_{\bphi}^x \frac{v(y)}{D(y)} \ud y \right\}, \quad q(x) = \frac{\sigma p(x)}{D(x)},
\end{equation}
where the integral on the definition of $p$ is taken along the unique path connecting the root $\bphi$ and the point $x \in \Gamma$. The functions $v$ and $D$ are defined on $\Gamma$ by taking the values $v_e$ and $D_e$ on edge $e$ respectively. Calculation of the resolvent of $\A$, involves inverting the operator $(\sigma - \A)$ for an arbitrary $\sigma >0$. For $\A$ given in \refp{A}, it follows that 
\begin{equation}\label{defL}
\L := \frac{p}{D}(\sigma - \A)
\end{equation}
has the familiar form of a Sturm-Luiville operator on each edge,
\begin{equation}
\L[f]_e = -(pf_e')' + qf_e.
\end{equation}
One is then interested in solving the problem
\begin{equation}\label{Lhf}
\L[f] = h, \quad f \in \E(\Gamma)
\end{equation}
where $\E(\Gamma)$ is the set of functions that are twice continuously differentiable inside each edge, satisfy internal node conditions \refp{contNode} and \refp{fluxNode}; and boundary conditions \refp{boundCond}. It can easily be shown that problem \refp{Lhf} has a unique solution if and only if it is non-degenerate, that is the only solution to the homogenous problem $\L[f] = 0$, $f \in \Dom(\L)$ is $f\equiv 0$.

Let $\mathcal G$ be some right inverse mapping of $\L$ namely $\L[\mathcal G[f]] = f$ for all admissible $f$. Then $\mathcal G$ is the Green's operator for problem problem \refp{Lhf}. Moreover, it will be shown that:

\begin{theorem}\label{ThmMain}
If problem \emph{(\ref{Lhf})} is non-degenerate, it has a unique Green's function. Namely, there exists a function $G:\Gamma \times \Gamma \to \R$ such that the solution to \emph{\refp{Lhf}}, $f = \mathcal{G}[h]$, is $f(x) = \int_\Gamma G(x,y) h(y) \ud y.$
\end{theorem}


The proof of the existence of the Green's function $G$ defined in Theorem \ref{ThmMain} follows from standard arguments. Here, the proof is obtained by simply giving an explicit formula for $G(x,y)$. Uniqueness of the Green's function follows from the hypothesis of non-degeneracy.   

In \cite{Pokornyi:2004p3263} the authors provide a proof of existence, uniqueness and a formula for $G$ for general graphs. The goal here is to present a new formula, both simpler and less expensive to compute, for the Green's function in the case $\Gamma$ is a tree graph. The techniques here are elementary and based on the classical Lagrange's method for Sturm-Liuoville problems (see for example \cite{Guenther:1996fk}).

The organization is as follows. The next section settles the notation and defines the class of Sturm-Liouville problems to be considered. Finally, section \ref{SectConstG} is devoted to the construction and the formula for the the Green's function.

\section{Sturm-Liouville problems on tree graphs}

\subsection{Tree graphs and functions}

By a tree graph we understand a finite collection of edges embedded in $\R^2$, joined with nodes and containing no loops. That is, for any two points $x,y$ in the graph, there exists one single path through the graph joining them. 
We assume that each edge $e$ of the graph allows a sufficiently smooth parametrization, contains no self-intersections, and is finite, therefore can be considered as the interval $e = (0,l_e)$. The collection of all graphs is denoted by $\Gamma$.  At each endpoint of an edge is located a node of $\Gamma$. The set of nodes is $N(\Gamma)$ and boldface is used to denote individual nodes. The graph, including its nodes, is denoted as $\bar\Gamma := \Gamma \cup N(\Gamma)$.

Points in $\Gamma$ are denoted by the pair $(e,x)$ with $0 < x < l_e$, or by single letters if specification of the edge is not crucial.  If $\n$ is a node, let $i(\n)$ denote the set of incident edges at $\n$, namely those for which $\n$ is an endpoint. Boundary nodes are those $\n$ with $\#i(\n)=1$. The set of all boundary nodes of $\Gamma$ is $\d \Gamma$. The set of internal nodes is $I(\Gamma) = N(\Gamma) \setminus \d \Gamma$.
Node $\n$ has $\#i(\n)\geq 1$ possible representations: for each $e \in i(\n)$, one either has $\n = (e,0)$ or $\n = (e,l_e)$. The representation of points in $\bar\Gamma$ is therefore dependent on the parametrization direction of its edges.

The value of a function $f:\Gamma \to \R$ at a point in $\Gamma$ is denoted as $f_e(x) = f(e,x)$. That is, $f_e$ is the restriction of $f$ to the edge $e$. For a node $\n \in \d \Gamma$ located at the endpoint of some edge $e$, $f(\n)$ denotes the appropriate one-sided limit of $f_e$. For an internal node $\n$ with $i(\n) = \{e_1,\dots,e_n\}$ the value $f_{e_1}(\n)$ denotes the one sided limit of $f_{e_i}$ as $x$ approaches the endpoint of $e_i$ at which $\n$ is located, $i=1,\dots,n$. If all these limits coincide, $f$ is said to be continuous at $\n$, and $f(\n)$ is defined as the common value.

We must also differentiate functions given on $\Gamma$. For a point $(e,x) \in \Gamma$ with $0<x<l_e$, the derivative $f'(e,x)=f'_e(x)$ is computed as the usual derivative of the restriction $f_e$ at $x$ according to the particular parametrization direction of $e$. A change in the orientation of the parametrization of the edge implies a sign change on $f_e'$. Note that the sign of $f_e''$ or $(pf_e')'$ remains unchanged. For a node $\n$ located at an endpoint of edge $e$, we introduce the \emph{boundary derivative} $f'^b_e$ as the derivative ``out of node $\n$ into edge $e$'': as if the parametrization of $e$ has $\n = (e,0)$. Boundary derivatives are useful because they make the following equality hold
\begin{equation}
\int_e (pf')' \ud x = p(l_e) f'^b(l_e) - p(0)f'^b(0)
\end{equation}
regardless of whether the integral is computed from $0$ to $l_e$, or from $l_e$ to $0$.

The space of functions that are $n$ times continuously differentiable in $\Gamma$, is denoted by $\C^n(\Gamma)$, $n=0,1,2,\dots$; $\C(\Gamma) := \C^0(\Gamma)$. Clearly, such spaces are identifiable with direct sums of the form $\bigoplus_{e} \C^n(0,l_e)$. The set $\C(\bar \Gamma)$ is composed of functions in $\C(\Gamma)$ that are also continuous at each node.

\subsection{Sturm-Liouville operators}\label{SecGeneralSL}

Let $p,q \in \C(\Gamma)$ be bounded with $\inf_{\Gamma} p(x) >0$. The object of this study is the following differential operator 
\begin{equation}\label{normalL}
\L[f]  := -(pf')' + qf, \quad f \in \D_p^2(\Gamma),
\end{equation}
where $\D_p^2(\Gamma)$ denotes the space of functions $f\in \C(\Gamma)$ such that $(pf')' \in \C(\Gamma)$. 

Green's functions for Sturm-Liouville operators are useful for solving more general problems than the one outlined on the introduction. For this more general treatment we follow \cite{Pokornyi:2004p3263}. 

If $\Gamma$ contains $m$ edges, then the dimension of $\mathcal N(\L) := \{u:\in \D^2_p(\Gamma);\; \L[u]=0\}$ is $2m$. A basis $\{\vp_1,\dots,\vp_{2m}\}$ for $\mathcal{N}(\L)$ can be found as follows. Let $e$ be the $i$-th edge and define $\vp_{2i-1}$ and $\vp_{2i}$ as the solutions to $-(p_e f')' + q_e f = 0$ on $e$ satisfying 
$$\vp_{2i-1}(0) = 1,\;\vp_{2i-1}'(0) = 0, \quad \vp_{2i}(0) = 0,\;\vp_{2i}'(0) = 1$$
extended to all of $\Gamma$ via $\vp_{2i-1}(\tilde e,x) = \vp_{2i}(\tilde e,x) = 0$, $0<x<l_{\tilde e}$ for all $\tilde e \neq e$.
Consider now a collection of $2m$ linear functionals $\{\l_i,\dots,\l_{2m}\}$ defined on $\D_p^2(\Gamma)$. The problem 
\begin{equation}\label{Lhf2}
f\in \D_p^2(\Gamma), \quad \L[f] =h, \quad \l_i[f] = c_i, \; i=1\dots,2m 
\end{equation}
will be uniquely solvable if and only if the homogenous problem
\begin{equation}\label{homo}
f\in \D_p^2(\Gamma), \quad \L[f] =0, \quad \l_i[f] = 0, \; i=1\dots,2m 
\end{equation}
has no solution except the trivial solution $f \equiv 0$. In this case we say that problem \refp{Lhf2} is \emph{non-degenerate}.

Non-degeneracy can be characterized as follows. Let $\Delta$ be the matrix defined by $\Delta_{i,j} = \l_i[\vp_j]$, $i,j=1,\dots,2m$. Non-degeneracy is therefore equivalent to $\det(\Delta) \neq 0$. In this case, the solution to problem \refp{Lhf2} can be written explicitly. Let $z$ be some solution to the semi-homogeneous problem
\begin{equation}\label{semihomo}
z\in \D_p^2(\Gamma), \quad \L[z] = h, \quad \l_i[z] = 0, \; i=1\dots,2m 
\end{equation}
then the solution $f = z+\sum_{i=1}^{2m} a_i \vp_i$ to problem \refp{Lhf2} must satisfy

\begin{equation}
\left[\begin{array}{cccc}
1			&\vp_1	&\dots	&\vp_{2m}   \\ \cline{2-4}
0			&\multicolumn{3}{|c|}{} \\ 
\vdots		&\multicolumn{3}{|c|}{\Delta}\\
0		&\multicolumn{3}{|c|}{}\\ \cline{2-4}
\end{array}\right]\left[ \begin{array}{c}
f \\ -a_1 \\ \vdots \\ -a_{2m}
\end{array}\right] 
=
\left[ \begin{array}{c}
z \\ -c_1 \\ \vdots \\ -c_{2m}
\end{array}\right].
\end{equation}
Hence, Cramer's rule gives the useful formula
\begin{equation}\label{homoSolved}
f = \frac{1}{\det(\Delta)}\,\det\!
\left[\begin{array}{cccc}
z		&\vp_1	&\dots	&\vp_{2m}   \\ \cline{2-4}
-c_1		&\multicolumn{3}{|c|}{} \\ 
\vdots		&\multicolumn{3}{|c|}{\Delta}\\
-c_{2m}		&\multicolumn{3}{|c|}{}\\ \cline{2-4}
\end{array}\right].
\end{equation}

\subsection{The physical problem}

Motivated by physical applications,  we now specialize to semi-homogenous Sturm-Liouville problems  where the functionals $\{\l_i: i=1,\dots,2m\}$ correspond to a particular choice of conditions at the nodes of $\bar\Gamma$.

Consider the operator $\L[f]$ acting on the set
\begin{equation}\label{E}
\E(\Gamma) = \D_p^2(\Gamma) \cap \C(\bar \Gamma)  \cap F_\rho(\Gamma) \cap B_D(\Gamma).
\end{equation}
Where $B_D(\Gamma)$ and $F_\rho(\Gamma)$ contain the boundary and weighted flux matching conditions respectively:
\begin{eqnarray}
B_D &=& \left\{f \in \C(\bar\Gamma): f(\n) = 0, \: \n\in \d\Gamma  \right\}\\ \label{BD}
F_\rho(\Gamma) &=& \left\{f \in \C^1(\Gamma): \sum_{e \in(\n)} \rho_{e} f_{e}'^b(\bs n) = 0, \: \n\in I(\Gamma)  \right\}. \label{Frho}
\end{eqnarray}
The function $\rho$ in \refp{Frho} is assumed constant on edges and strictly positive. Other boundary conditions types than Dirichlet -- like those in \refp{boundCond}-- can be considered without major changes to the arguments that follow. 

The conditions encoded in $\E(\Gamma)$ can be cast in terms of linear functionals: let $\n \in I(\Gamma)$  with $i(\n) = \{e_1,\dots,e_k\}$, and define the functionals  
\begin{eqnarray}
\tilde \l_{\n,i}[f] &=& f_{e_{i+1}}(\n) - f_{e_i}(\n), \quad i=1,\dots,k-1,\\ \label{funcCont}
\tilde \l_{\n,k}[f] &=& \sum_{i=1}^k \rho_{e_i} f_{e_i}'^b(\bs n). \label{funcFlux}
\end{eqnarray}
For a boundary node $\n$ located at the endpoint of edge $e$, define simply
\begin{equation} \label{funcBD}
\tilde \l_{\n}[f] = f_e(\n).
\end{equation}
Relabeling gives a collection of $2m$ functionals $\{\l_1,\dots,\l_{2m}\}$, such that the problem of finding $f \in \E(\Gamma)$ satisfying $\L[f] = h$ can be written as the semi-homogenous problem \refp{semihomo}.


\section{Construction of the Green's function}\label{SectConstG}
The goal is to arrive at a formula for the solution to problem \refp{semihomo}. The solution to the associated non-homogenous problem will then follow from \refp{homoSolved}.

\begin{definition}\label{DefG}
A Green's function for operator $\L$ is a function $G:\Gamma \times \Gamma \to \R$ such that for all $h \in \text{Ran}(\L)$, the solution to problem \refp{semihomo} is given by 
\begin{equation}
f(x) = \int_\Gamma G(x,y) h(y) \ud y.
\end{equation}
The operator $\mathcal G:h \mapsto \int_\Gamma G(x,y) h(y) \ud y$ is called the Green's operator.
\end{definition}

The first step is elementary and consists on verifying properties of the Wronskian of functions on $\Gamma$. For $f,g \in \C^1(\Gamma)$ the Wronskian $W[f,g]$ is defined on an edge $e$ of $\Gamma$ as
\begin{equation}
W[f,g]_e = f_eg_e'-g_ef_e'
\end{equation}

\begin{lemma} \label{LemmaW} Let $f,g,h$ be functions in $\C^1(\Gamma)$.
\begin{enumerate}
\item \label{a} If $f,g \in \D_p^2(\Gamma)$ then Lagrange's identity holds on each edge $e$,
\begin{equation}\label{Lagranges}
f_e\L[g]_e - g_e \L[f]_e = -\deriv{}{x}(p_eW[f,g]_e).
\end{equation}
\item \label{b} If $f,g \in B_D$, then $W[f,g] \in B_D(\Gamma)$.
\item If $f,g \in \C(\bar\Gamma) \cap F_\rho(\Gamma)$, then $\sum_{e\in i(\n)} \rho_{e} W[f,g]_e(\bs n) =0$ for all $\n \in I(\Gamma)$. Here, the derivatives in the definition $W$ at $\n$ are replaced by boundary derivatives.
\item $hW[f,g] - fW[h,g] = gW[f,h]$.
\item If $f,g \in \D_p^2(\Gamma)$ with $\L[f] = \L[g] = 0$ on some edge, then $pW[f,g]$ is constant there.
\end{enumerate}
\end{lemma}

\begin{proof}
Statement (\textit{a}) follows from a simple calculation and (\textit{b}) is obvious. For (\textit{c}), it suffices to use continuity and change derivatives to boundary derivatives,
$$\sum_{e\in i(\n)} \rho_{e} W[f_e,g_e](\bs n) = f(\n) \sum_{e\in i(\n)} \rho_{e} g_e'^b(\bs n) + g(\n) \sum_{e\in i(\n)} \rho_{e} f_e'^b(\bs n) = 0.$$
 (\textit{d}) is obtained by rearranging terms. To prove (\textit{e}), compute $(p_eW[f,g]_e)' = f_e(p_eg_e')' - g_e(p_e f_e')'$, use $(p_ef_e')'=-q_ef_e$ and $(p_eg_e')'=-q_eg_e$ to finally get $(p_eW[f,g]_e)'=0$. 
\end{proof}

For the following definition, and subsequent formulas, assume without loss of generality that the parametrization of $\Gamma$ is such that for all $\n \in \d \Gamma$, we have $\n = (e,l_e)$ where $e$ is the edge $\n$ belongs to. 

\begin{definition}
Refer to figure \ref{FigTrees}. Let $(e,x) \in \Gamma$, 
\begin{enumerate}
\item The two connected components of $\bar\Gamma \setminus \{(e,x)\}$ are denoted $\bar\Gamma(e,x)$ and $\bar\Lambda(e,x)$ respectively. The point $(e,x)$ is adjoined as a boundary node to $\bar\Gamma(e,x)$ and $\bar\Lambda(e,x)$. By convention, $\bar\Gamma(e,x)$ is taken as the tree that contains the node $(e,0)$. There is an edge denoted by $e$ in both $\bar \Gamma(e,x)$ and $\bar \Lambda(e,x)$; it is parametrized as the intervals $(0,x)$ and $(x,l_e)$ respectively.  

\item For an edge $e$, $\bar\Gamma(e) := \cup_{x\in(0,l_e)}\bar\Gamma(e,x)$, and $\bar\Lambda(e) := \cup_{x\in(0,l_e)}\bar\Lambda(e,x)$. 
\item As in the case of the full tree, $\Gamma(e,x)$ -- without the bar -- denotes the collection of points inside edges of $\bar\Gamma(e,x)$. Similarly for $\Lambda(e,x), \Gamma(e), \Lambda(e)$.
\item $\E_0(\Gamma(e))$ is the set of functions $f \in \D_p^2(\Gamma(e)) \cap \C(\bar\Gamma(e)) \cap F_\rho(\Gamma(e))$ such that $f(\n) = 0$ for $\n \in \d\Gamma(e) \setminus \{(e,l_e)\}$. 
\item Similarly, $\E_0(\Lambda(e))$ is comprised of functions $f \in \D_p^2(\Lambda(e)) \cap \C(\bar \Lambda(e)) \cap F_\rho(\Lambda(e))$ such that $f(\n) = 0$ for $\n \in \d \Lambda(e) \setminus \{(e,0)\}$. 
\end{enumerate}
\end{definition}

\begin{center}
\begin{figure}
\includegraphics[scale = 1]{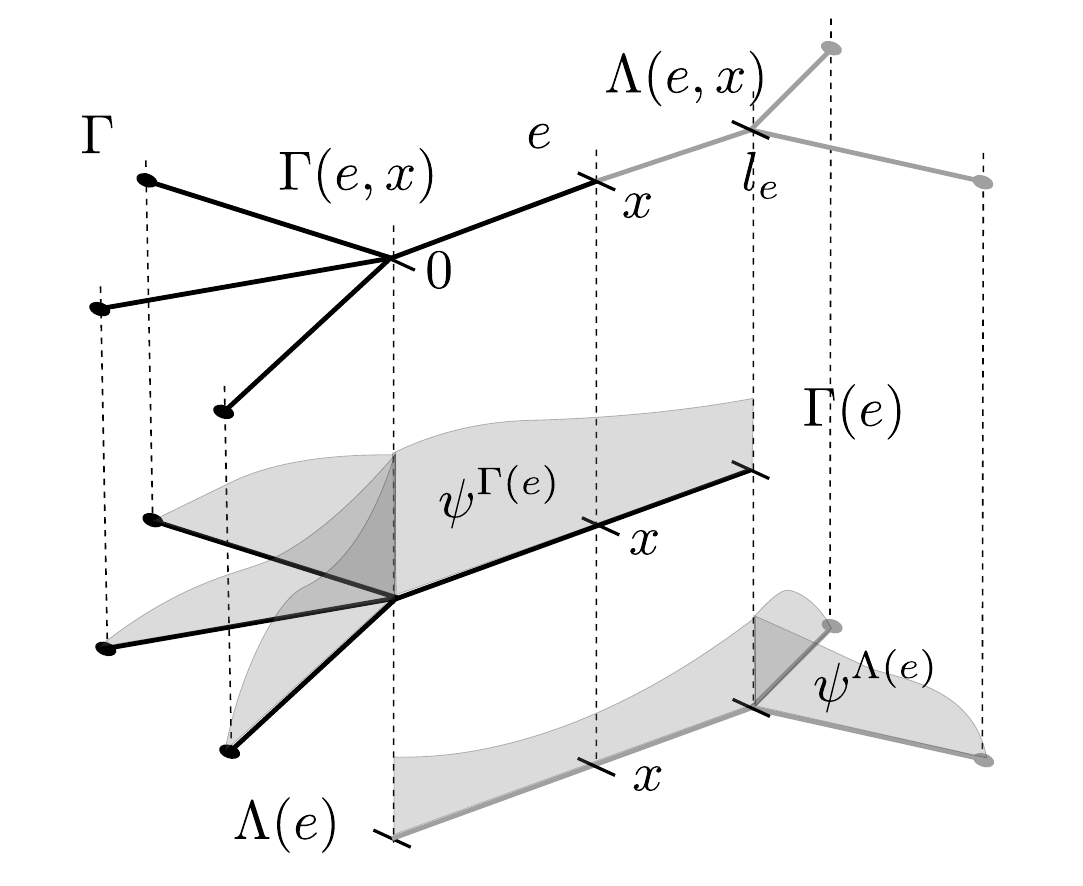}
\caption{Schematic representation of a tree $\Gamma$ with six edges. For $e$ and $x$ as shown, the black sub-tree on the upper figure in $\Gamma(e,x)$, the gray sub-tree is $\Lambda(e,x)$. The middle and lower figures depict trees $\Gamma(e)$ and $\Lambda(e)$ respectively with a schematization of the functions $\psi^{\Gamma(e)}$ and $\psi^{\Lambda(e)}$.}\label{FigTrees}
\end{figure}
\end{center}

\begin{lemma}\label{lemmaPsis}
Let $e$ be a fixed edge. If problem \emph{\refp{semihomo}} is non-degenerate, then there exists solutions $\psi^{\Gamma(e)} \in \E_0(\Gamma(e))$, $\psi^{\Lambda(e)} \in \E_0(\Lambda(e))$ to $\L \psi^{\Gamma(e)} = 0$ on $\Gamma(e)$,  and $\L \psi^{\Lambda(e)} = 0$ on $\Lambda(e)$. These functions can further be chosen so that $p W[\psi^{\Gamma(e)},\psi^{\Lambda(e)}] = -1$ on $e$.
\end{lemma}
\begin{proof}
If $e$ contains no nodes in $\d\Gamma$, choose a node $\n' \notin \d \Gamma(e)$ and let $e'$ be its edge. If $e$ contains a node in $\d \Gamma$, make $\n'$ equal to that node, and $e' =e$.  Rearrange the basis $\{\vp_1,\dots, \vp_{2m}\}$ so that $\vp_1$ and $\vp_2$ are supported on $e'$, and the functionals $\{\l_1,\dots, \l_{2m}\}$ so that $\l_1[f] = f(\n')$. Let $\Delta$ be the matrix defined in section \refp{SecGeneralSL}. By the nondegeneracy of problem \refp{semihomo},  there exist a solution $a=(a_1,\dots,a_{2m})$ to $\Delta a = \varepsilon^{(1)}$, where $\varepsilon^{(1)}$ denotes the $\R^{2m}$ vector that has one in the first coordinate, and zero elsewhere. The function $\psi = \sum_{i=1}^{2m} a_i \vp_i$ is a solution in $\D_p^2(\Gamma) \cap \C(\bar \Gamma)  \cap F_\rho(\Gamma)$ to $\L[\psi] = 0$ on all of $\Gamma$ and such that that $\psi(\n) = 0$ for $\n \in \d(\Gamma) \setminus \{\n'\}$. The restriction of $\psi$ to $\Gamma(e)$ serves as the required function $\psi^{\Gamma(e)}$. A similar construction applies for $\psi^{\Lambda(e)}$. By Lemma \ref{LemmaW}, $p W[\psi^{\Gamma(e)},\psi^{\Lambda(e)}]$ is constant on $e$, and the desired normalization can be achieved if this constant is not zero. Assume on the contrary that $p W[\psi^{\Gamma(e)},\psi^{\Lambda(e)}]=0$ on $e$. Since the Wronskian vanishes, there is $k \neq 0$ such that $\psi^{\Gamma(e)}_e = k \psi^{\Lambda(e)}_e$. The function $f := \psi^{\Gamma(e)} \IND{\Gamma(e)} + k \psi^{\Lambda(e)} \IND{\Gamma(e)^c}$ would then be a solution to the homogenous problem \refp{homo} violating the assumption of non-degeneracy.      
\end{proof}

\begin{remark}
The computation of the $\psi^{\Lambda(e)}, \psi^{\Gamma(e)}$ can be performed quite inexpensively. For a boundary node $\n' = (e',l_{e'}) \in \d \Gamma$, the solution $\psi$ constructed in the proof of lemma \ref{lemmaPsis} can be restricted to define $\psi^{\Gamma(e)}$ for all nodes $e$ such that either $e=e'$ or $e'$ does not belong to $\Gamma(e)$. Similarly it can be used to define $\psi^{\Lambda(e)}$ for all nodes $e$ such that $e'$ does not belong to $\Lambda(e)$. This implies that the linear system $\Delta a = \varepsilon^{(1)}$ has to be solved only $\#\d\Gamma$ times.
\end{remark}

The specific form of the Green's function can now be written.

\begin{theorem}\label{TheoGreen}
Assume problem \refp{semihomo} is non-degenerate. The following function is a Green's function for operator $\L$,
\begin{equation}\label{G}
G(x,y) = \frac{1}{\rho_e} \times
\begin{cases}
\psi^{\Gamma(e)}(y) \: \psi^{\Lambda(e)}(x), & y \in \Gamma(e,x)\\
\psi^{\Lambda(e)}(y) \: \psi^{\Gamma(e)}(x), & y \in \Lambda(e,x)
\end{cases}, \quad x \in e.
\end{equation}
Moreover, this function is unique in the class of continuous functions on $\Gamma \time \Gamma$ that are continuous with respect to the first variable.
\end{theorem}

\begin{proof}
Let $h \in \text{Ran}(\L)$, and $f \in \E(\Gamma)$ a solution to $\L[f] = h$. Fix an edge $e$, and $x\in e$. Applying  Lagrange's identity \refp{Lagranges} for $\psi^{\Gamma(e)}$ and $f$ and integrating over $\Gamma(e,x)$ with respect to the measure $\rho$ gives
$$\int\limits_{\Gamma(e,x)} \psi^{\Gamma(e)} h \ud \rho = - \!\!\sum_{a \subset \Gamma(e,x)} \!\! \left( p_aW[\psi^{\Gamma(e)},f]_a \, \rho_a \right|_0^{l_a},$$
where the sum on the right hand side is taken over all edges $a$ of $\Gamma(e,x)$. Parts (\textit{b}) and (\textit{c}) of lemma \ref{LemmaW} ensure that all terms in the sum cancel except for the value at $(e,x)$,
\begin{equation}\label{int1}
\int\limits_{\Gamma(e,x)} \psi^{\Gamma(e)} h \ud \rho = - p_e(x) \rho_e W[\psi^{\Gamma(e)},f]_e(x).
\end{equation}
Similarly, Lagrange's identity for $\psi^{\Lambda(e)}$ and $f$, gives
\begin{equation}\label{int2}
\int\limits_{\Lambda(e,x)} \psi^{\Lambda(e)} h \ud \rho =  p_e(x) \rho_e W[\psi^{\Lambda(e)},f]_e(x).
\end{equation}
Multiply equations \refp{int1} and \refp{int2} by $\psi^{\Lambda(e)}(x)$ and $\psi^{\Gamma(e)}(x)$ respectively, add the resulting equations, and apply part (\textit{d}) of lemma \ref{LemmaW} to the right hand side of the result. Finally, since $p W[\psi^{\Gamma(e)},\psi^{\Lambda(e)}] = -1$ on $e$,
\begin{equation}
\begin{split}
\int\limits_{\Gamma(e,x)} \!\! \psi^{\Lambda(e)}(x) \psi^{\Gamma(e)}(y) \, h(y) \ud \rho(y)  \;\; + & \\
 \int\limits_{\Lambda(e,x)} \!\! \psi^{\Gamma(e)}(x) &\psi^{\Lambda(e)}(y) \,  h(y) \ud \rho(y) \;=\; f_e(x) \rho_e.
\end{split}
\end{equation}
Since $\Gamma$ is a disjoint union of $\Gamma(e,x)$ and $\Lambda(e,x)$, the function $G(x,y)$ defined in \refp{G} satisfies definition \refp{DefG}. Let $h \in C(\Gamma)$ arbitrary. It will be establised now that $h \in \text{Ran}(\L)$ simply by showing that $f:= \mathcal G h$ solves $\L[f] = h$. Write
\begin{equation}\label{computeLf}
\begin{split}
f(x)  \;  = \quad &  \psi^{\Lambda(e)}(x) \!\! \int\limits_{\Gamma(e) \smallsetminus e}  \!\! \psi^{\Gamma(e)} h \ud \rho + 
		 \psi^{\Gamma(e)}(x)\!\! \int\limits_{\Lambda(e) \smallsetminus e} \!\! \psi^{\Lambda(e)} h \ud \rho \\
	   & + \; \psi^{\Lambda(e)}(x) \int_0^x \psi^{\Gamma(e)} h \ud\rho + \psi^{\Gamma(e)}(x) \int_x^{l_e} \psi^{\Lambda(e)} h \ud\rho.
\end{split}
\end{equation}
Applying $\L$ to the first two terms in \refp{computeLf} gives zero since $\L[ \psi^{\Lambda(e)}] = \L[ \psi^{\Lambda(e)}] = 0$. A routine calculation finally shows that 
$$ \L[f] = -h p W[\psi^{\Gamma(e)}, \psi^{\Lambda(e)}]_e + \L[\psi^{\Lambda(e)}] \int_0^x \psi^{\Gamma(e)} h \ud\rho
+\L[\psi^{\Gamma(e)}] \int_x^{l_e} \psi^{\Lambda(e)} h \ud\rho $$
which yields $\L[f] = h$. Lastly, the non-degeneracy of problem \refp{semihomo} and the fact that $\text{Ran}(\L) = \C(\Gamma)$, imply the uniqueness of $G$ as stated in the theorem.
\end{proof}
  
\begin{remark}
The construction of the Green's function in Theorem \ref{TheoGreen} has one particular important advantage over the one proposed by \citep{Pokornyi:2004p3263}. In that work, $G(x,y)$ is given as
\begin{equation}\label{Pokornyi}
G(x,y) = H(x,y) - \sum_{i=1}^{2m} \l_i[H(\cdot,y)] \eta_i(x)
\end{equation}
where $H(x,y)$ is equal to the Green's function of operator $\L$ on $(0,l_e)$ if $x,y \in e$, and equal to zero whenever $x$ and $y$ belong to different edges. The functions $\eta_i$ are solutions to $\L[\eta_i] = 0$, $\l_j[\eta_i] = \delta_{ij}$. Note that this formula requires solving $\Delta a = \varepsilon^{(1)}$ a total of $2m$ times to compute $G(x,y)$ at single pair of points $(e_x,x)$, $(e_y,y)$ of $\Gamma$. Via formula \refp{G}, one needs only the functions $\psi^{\Lambda(e_x)}$ and $\psi^{\Gamma(e_x)}$ and therefore, the system $\Delta a = \varepsilon^{(1)}$ must be solved only twice. On the other hand, formula \refp{Pokornyi} has the advantage of using $H$, which is a diagonal fundamental solution to $\L[f]=h$.
\end{remark}  
  
\bibliographystyle{plainnat}  
\bibliography{biblio}

\end{document}